\documentclass{IEEEtran}
\usepackage{cite}
\usepackage{amsmath,amssymb,amsfonts,amsthm}
\usepackage{algorithmic}
\usepackage{graphicx}
\usepackage{textcomp}

\def\BibTeX{{\rm B\kern-.05em{\sc i\kern-.025em b}\kern-.08em
    T\kern-.1667em\lower.7ex\hbox{E}\kern-.125emX}}

    \newtheorem{theorem}{Theorem}
\newtheorem{remark}[theorem]{Remark}
\newtheorem{lemma}[theorem]{Lemma}

\usepackage{algorithmic}
\usepackage{algorithm}

\newcommand{\N}{\mathbb{N}}

\newcommand{\A}{\mathcal{A}}
\newcommand{\B}{\mathcal{B}}
\newcommand{\R}{\mathbb{R}}
\newcommand{\Kvu}{K^{vu}}
\newcommand{\Kvv}{K^{vv}}

\begin{document}
\title{Event-triggered boundary control of $2\times2$ semilinear hyperbolic systems}
\author{Timm Strecker, Michael Cantoni \IEEEmembership{Member, IEEE}, and Ole Morten Aamo, \IEEEmembership{Senior Member, IEEE}
\thanks{This work was supported by the Australian Research
Council under Grant LP160100666. }
\thanks{Timm Strecker and Michael Cantoni are with the Department of
Electrical and Electronic Engineering, The University of Melbourne,
Melbourne, VIC 3052, Australia (e-mail: timm.strecker@unimelb.edu.au;
cantoni@unimelb.edu.au). }
\thanks{Ole Morten Aamo is with the Department of Engineering Cybernetics,
Norwegian University of Science and Technology (NTNU), N-7491
Trondheim, Norway (e-mail: aamo@ntnu.no).}
}

\maketitle

\begin{abstract}
We present an event-triggered boundary control scheme for hyperbolic systems. The trigger condition is based on predictions of the state on determinate sets, where the control input is only updated if the predictions deviate from the reference by a given margin. Closed-loop stability and absence of Zeno behaviour is established analytically.
For the special case of linear systems, the trigger condition can be expressed in closed-form as an $L_2$-scalar product of kernels with the distributed state. The presented controller can also be combined with existing observers to solve the event-triggered output-feedback control problem. A numerical simulation demonstrates the effectiveness of the proposed approach.
\end{abstract}


\section{Introduction}

We consider $2\times 2$ semilinear hyperbolic systems  with one boundary actuator of the form
\begin{align}
u_t(x,t)&= -\lambda^u(x)\,u_x(x,t) + f^u(w(x,t),x), \label{u_t}\\
v_t(x,t)&= \phantom{-}\lambda^v(x)\,v_x(x,t) + f^v(w(x,t),x), \label{v_t}\\
u(0,t) &= g(v(0,t),t),\label{uBC}\\
v(1,t)&= U(t),\label{vBC}\\
w(\cdot,0) &= w_0,\label{w_0}
\end{align}
where $x\in[0,1]$, $t\geq 0$, $w(x,t) = \left(u(x,t),v(x,t)\right)$, subscripts denote partial derivatives, $U(t)$ is the control input, and $w_0$ the initial condition. 

Systems  of  form (\ref{u_t})-(\ref{w_0}) model a range of 1-d transport processes including gas or fluid flow through pipelines, open channel flow, traffic flow, electrical transmission lines, and blood flow in arteries \cite{bastin2016book}. Consequently, the control and observer design for such systems has received much attention.

Static controllers that achieve assymptotic convergence are designed using dissipative boundary conditions in \cite{greenberg1984effect} and using  control Lyapunov functions in \cite{coron2007strict}. The exact-finite time controllability of such systems is analysed in \cite{cirina1969boundary,li2003exact}. For the special case of linear systems,  backstepping  has become a popular method for designing feedback controllers that achieve this kind of performance. See, e.g., \cite{vazquez2011backstepping,aamo2013disturbance,vazquez2016bilateral,auriol2016two,deutscher2017finite,
auriol2016minimum,di2018stabilizationPDEODE,auriol2020seriesinterconnection} for a range of results using different configurations and stabilization or tracking as the objective. For semilinear and quasilinear systems, the same control performance has been achieved using a predictive approach \cite{strecker2017output,strecker2017seriesinterconnections,strecker2017twosided,strecker2021quasilinear2x2,strecker2021quasilinear_PDE_ODE}.

Recently, there has been interest in the  event-triggered boundary control of (linear) hyperbolic systems \cite{espitia2017event,espitia2020observer,wang2021adaptive,wang2020event}. The related topic of sampled-data control is analysed in \cite{davo2018stability}. In event-triggered  control, the control inputs are held piecewise constant and are updated only when needed, instead of continuously in time or in a periodic fashion. One of the main benefits of such an approach is that it reduces wear and tear on the physical actuators, which is of  interest in many applications.
See also \cite{heemels2012introduction} for event-triggered control of ordinary differential equations (ODE).

The contribution of this paper is an event-triggered implementation of the predictive approach from \cite{strecker2017output}. For clarity of presentation, the focus is on the stabilisation of semilinear hyperbolic systems using state-feedback  control, although the approach can readily be adapted to tracking problems, output-feedback control and  quasilinear systems.  Similarly, the extensions to $n+1$ systems (that is, vector-valued $u$ as considered in \cite{dimeglio2013stabilization} for linear systems), and to time-varying $f^u$ and $f^v$ are straightforward.

The paper is structured as follows. Section \ref{sec:preliminaries} contains several preparations including the precise model assumptions, sufficient conditions for stability and well-posedness and the continuous-time implementation of the controller. The event-triggered controller is presented in Section \ref{sec:event-triggered control}, including a closed-form implementation (i.e., without the need to solve PDEs online) for the special case of linear systems in Section \ref{sec:linear closed form}. Simulation examples are shown in Section \ref{sec:example} and concluding remarks are given in Section \ref{sec:conclusions}.

\section{Preliminaries}\label{sec:preliminaries}
\subsection{Assumptions}
We consider broad solutions to system (\ref{u_t})-(\ref{w_0}), which are a type of weak solution defined by integrating (\ref{u_t})-(\ref{v_t}) along  characteristic lines and solving the obtained integral equations \cite{bressan2000hyperbolic}. Well-posedness for broad solutions can be shown under the following assumptions.

 The speeds $\lambda^u,\lambda^v\in L^{\infty}([0,1])$ are assumed to be bounded from below and above  by the finite values
\begin{align}
\sup_{x\in[0,1]} \left( \min\{ \lambda^u(x),\,\lambda^v(x) \}^{-1}\right) &= k_{\Lambda^{-1}},\\
\sup_{x\in[0,1]}  \max\{ \lambda^u(x),\,\lambda^v(x) \} &= k_{\Lambda}
\end{align}
 The nonlinear functions $f^u$, $f^v$, and $g$ are globally Lipschitz-continuous in the state arguments, i.e.,
\begin{equation}
 \operatorname*{ess\,sup}_{x\in[0,1],w\in\R^2}\max\left\{  \|\partial_w f^u(w,x) \|_{\infty} , \|\partial_w f^v(w,x) \|_{\infty}  \right\}\ =l_F,  \label{boundF_w}
  \end{equation}
 \begin{equation}
 \operatorname*{ess\,sup}_{v \in\R,t\geq0}  |\partial_v g(v,t) | = l_{g},  \label{bound g0_v}
\end{equation}
where $\partial$ denotes partial derivatives\footnote{By Rademacher's theorem, the partial derivatives of Lipschitz-continuous functions exist almost everywhere \cite[Theorem 2.8]{bressan2000hyperbolic}.} and $l_F$ and $ l_{g}$ are the finite Lipschitz constants. We assume further that the origin is an equilibrium, i.e.,  
\begin{align}
f^u(0,x)& = f^v(0,x) =0 & \text{ for all } &x\in[0,1],\label{F(0)=0}  \\
 g(0,t)&= 0 & \text{ for all } &t\geq 0.  \label{g(0)=0}
\end{align} 
 The initial  conditions are assumed  bounded with $w_0\in L^{\infty}([0,1],\R^2)$.

\begin{remark}
The  global Lipschitz condition is restrictive, but  made here to simplify some of the proofs. For locally Lipschiz-continuous data, similar local results can be shown. See also \cite{li2003exact} for local controllability results.
\end{remark}

\subsection{Characteristic lines and determinate sets}
\begin{figure}[htbp!]\centering
\includegraphics[width=.7\columnwidth]{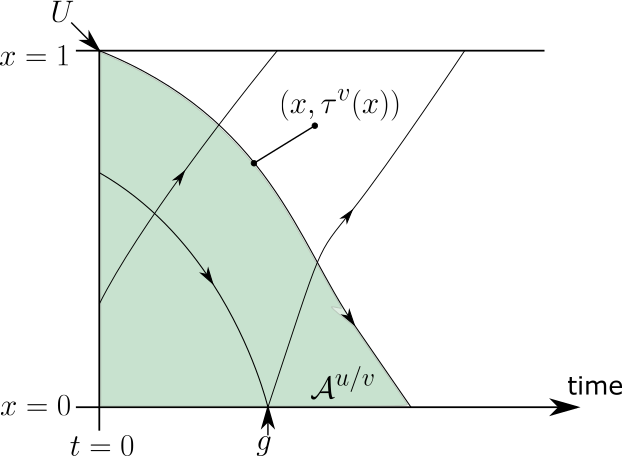}
\caption{Characteristic lines for system (\ref{u_t})-(\ref{w_0}) going upwards ($u$) and downwards ($v$). The area shaded in green indicates the determinate sets $\A^u(0)$ (including the line $(x,\tau^v(x))$) and $A^v(0)$ (excluding the line $(x,\tau^v(x))$) .}
\label{fig:characteristic lines}
\end{figure}
The effect of the control input $U(t)$ propagates through the domain $x\in[0,1]$ with finite speed $\lambda^v$. More precisely, the control input $U(t)$ entering at the boundary $x=1$ at time $t$ has an effect  on the state at some location $x$ only after a delay of $\tau^v(x)$, where
\begin{align}
\tau^u(x) &= \int_0^x \frac{1}{\lambda^u(\xi)}d\xi, & \tau^v(x) &= \int_x^1 \frac{1}{\lambda^v(\xi)}d\xi. \label{tau}
\end{align}
Consequently, the state on the determinate set  can be predicted based on the current state at time $t$, $w(\cdot,t)$, where the determinate sets are given by
\begin{align}
\A^u(t)&=\{(x,s):\,x\in[0,1],\,s\in[t,t+\tau^v(x)]\}, \\
\A^v(t)&=\{(x,s):\,x\in[0,1],\,s\in[t,t+\tau^v(x))\}.
\end{align}
\begin{lemma}\label{lemma determinate set}
For any $t\geq 0$, the Cauchy problem consisting of (\ref{u_t})-(\ref{uBC}) with $w(\cdot,t)$ as the initial condition at time $t$, has a unique solution $u(x,s)$ for $(x,s)\in\A^u(t)$ and $v(x,s)$ for $(x,s)\in\A^v(t)$.
Moreover, there exists  constant $c_1>0$ (depending on the model data) such that 
\begin{align}
\sup_{(x,s)\in\A^u(t)}|u(x,s)|&\leq c_1 \|w(\cdot,t)\|_{\infty}, \label{bound 1 u}\\
\sup_{(x,s)\in\A^v(t)}|v(x,s)| &\leq c_1 \|w(\cdot,t)\|_{\infty}.  \label{bound 1 v}
\end{align}
\end{lemma}
\begin{proof}
The proof is given in \cite[Appendix A]{strecker2017output}. See also \cite{bressan2000hyperbolic,li2000semi} for a general discussion of determinate sets.
\end{proof}
\begin{remark}
One convenient implementation for determining the solution of $u(x,s)$ on $\A^u(t)$, including the solution on the line $u(x,t+\tau^v(x))$, $x\in[0,1]$, is to solve the system consisting of (\ref{u_t})-(\ref{vBC}) with initial condition $w(\cdot,t)$ and an arbitrary input $U$ (e.g., $U\equiv \lim_{x\rightarrow 1} v(x,t)$)  over the larger rectangular domain $[0,1]\times[t,t+\tau^v(0)]$. The solution of $u$ on $\A^u(t)$ can then be selected from the larger solution where, as seen in Lemma \ref{lemma determinate set}, it is independent of the choice for $U$.
\end{remark}

\subsection{Sufficient condition for convergence}

\begin{figure}[htbp!]\centering
\includegraphics[width=.9\columnwidth]{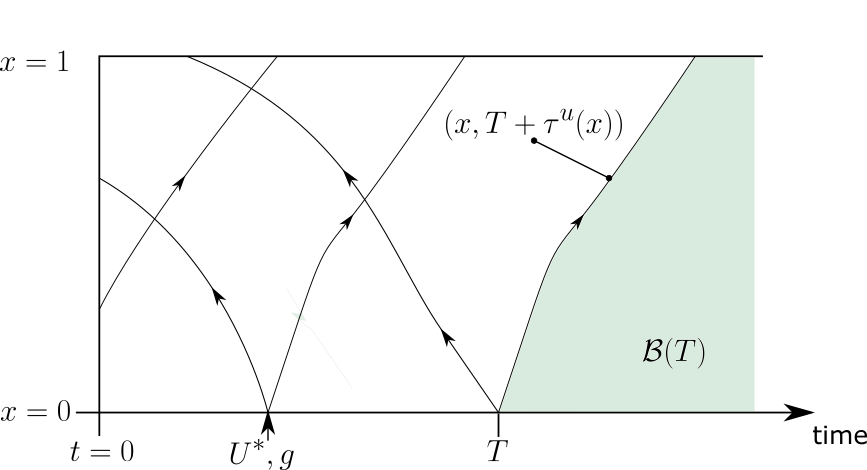}
\caption{Characteristic lines for system (\ref{u_x})-(\ref{u(0) x-direction}).}
\label{fig:characteristic lines x-direction}
\end{figure}

Following \cite{strecker2017output}, \cite{strecker2021quasilinear2x2}, we exploit that convergence to the origin is much easier to characterize via conditions on the uncontrolled boundary value $v(0,\cdot)$. In particular, we introduce the virtual input $U^*(t)$, which is the target value for $v(0,\cdot)$.  We then reverse the roles of $t$ and $x$ in (\ref{u_t}) and (\ref{v_t}) to formulate the system as a PDE in the positive $x$-direction for $x\in[0,1]$.
\begin{align}
u_x(x,t)&= - \frac{1}{\lambda^u(x,t)}\,u_t(x,t) + \frac{f^u(w(x,t),x)}{\lambda^u(x)}, \label{u_x}\\
v_x(x,t)&= \phantom{-} \frac{1}{\lambda^v(x,t)}\,v_t(x,t) - \frac{f^v(w(x,t),x)}{\lambda^v(x)},\label{v_x}\\
v(0,t) &= U^*(t), \label{v(0) x-direction}\\
u(0,t) &= g(U^*(t),t). \label{u(0) x-direction}
\end{align}
See also Figure \ref{fig:characteristic lines x-direction} for the characteristic lines of the transformed system (\ref{u_x})-(\ref{u(0) x-direction}).
Using the same methodology underlying Lemma \ref{lemma determinate set}, it is possible to show the following:
\begin{lemma}\label{lemma determinate set x direction}
For any $T\geq \tau^v(0)$, the transformed system (\ref{u_x})-(\ref{u(0) x-direction}) with $U^*(t)$ restricted to $t\geq T$,  has  unique solution  on the set 
\begin{equation}
\B(T) = \{(x,s):\, x\in[0,1],\,s\geq T+\tau^u(x)\},  \label{B}
\end{equation}
independently of the initial condition $w_0$. Moreover, there exists  constant $c_2>0$ (depending on the model data) such that 
\begin{equation}
\sup_{(x,s)\in\B(T)} \|w(x,s)\|_{\infty} \leq c_2\,\sup_{s\geq T} |U^*(s)|.  \label{bound 2}
\end{equation}
\end{lemma}
Note that  $[0,1]\times[T+\tau^u(1),\infty) \subset\B(T)$. Therefore, (\ref{bound 2}) implies that if $U^*(t)=0$ for all $t\geq T$, then $w(\cdot,t) = 0$ for all $t\geq T+\tau^u(1)$.

\subsection{Continuous-time control design}\label{sec:continuous time controller}
The idea in \cite{strecker2017output} is to design $U(t)$ such that the future boundary value $v(0,t+\tau^v(0))$ becomes equal to the virtual input $U^*(t+\tau^v(0))$. By (\ref{bound 2}), if $v(0,t+\tau^v(0)) = U^*(t+\tau^v(0)) = 0$ for all $t\geq T$, the state reaches the origin by time $T + \tau^v(0) + \tau^u(1)$.  

Defining the future state on the characteristic line,
\begin{equation}
\bar{w}(x,t) =\left(\bar{u}(x,t),\bar{v}(x,t) \right) = w(x,t+\tau^v(x)),  \label{w_bar}
\end{equation}
the following relationship between $U$ and $v(0,t+\tau^v(0,t))$ is derived in  \cite[Theorem 2]{strecker2017output}.
\begin{lemma}
For given $\bar{u}$, the state $\bar{v}$ satisfies the ODE
\begin{align}
\bar{v}_x(x,t)&= -\frac{f^v((\bar{u}(x,t),\bar{v}(x,t)),x)}{\lambda^v(x)}, \label{vbar_x}\\
\bar{v}(1,t) &= U(t). \label{vbar_BC}
\end{align}
\end{lemma}
Crucially, (\ref{vbar_x}) is an ODE in the $x$-direction with no time-dynamics. Consequently, it can be solved in either $x$-direction. In (\ref{vbar_x})-(\ref{vbar_BC}), it is solved in  the negative $x$ direction, so that $\bar{v}(0,t)$ is a function of $U(t)$. As originally explored in \cite{strecker2017output}, the alternative  is to start with the desired $U^*(t+\tau^v(0))$ and solve a copy of (\ref{vbar_x}) in the \emph{positive} $x$-direction, i.e., backwards relative to how the actual input $U$ propagates through the domain. By setting $U^*(t)\equiv 0$, the closed-loop system converges to the origin in finite (minimum) time.

\begin{theorem}\label{theorem continuous time}
Consider the system consisting of (\ref{u_t})-(\ref{w_0}) in closed loop with $U(t)$ at each $t\geq 0$ set according to  the following algorithm:
%
\begin{enumerate}
\item Predict $\bar{u}(\cdot,t)$ by solving (\ref{u_t})-(\ref{uBC}) with $w(\cdot,t)$ as initial condition over the domain $\A^u(t)$;
\item Solve the target system 
\begin{align}
\bar{v}^*_x(x,t)&= -\frac{f^v((\bar{u}(x,t),\bar{v}^*(x,t)),x)}{\lambda^v(x)}, \label{vbar*_x}\\
\bar{v}^*(0,t) &= 0; \label{vbar*_BC}
\end{align}
\item Set $U(t)=\bar{v}^*(1,t)$.
\end{enumerate}
Then  $w(\cdot,t)=0$ for all $t\geq \tau^v(0) + \tau^u(1)$.
\end{theorem}
\begin{proof}
With $\bar{u}(\cdot,t)$ determined uniquely by the prediction in step 1),  (\ref{vbar_x}) and (\ref{vbar*_x}) are equivalent ODEs in $\bar{v}(\cdot,t)$ and $\bar{v}^*(\cdot,t)$, respectively. Since these ODEs have unique solution under the given assumptions, $\bar{v}(0,t)=\bar{v}^*(0,t) = 0$ if and only if $U(t) = \bar{v}(1,t) = \bar{v}^*(1,t)$.  That is, the construction in steps (1)-(3) ensures  $v(0,t)=0$ for all $t\geq \tau^v(1)$. Convergence to the origin then follows from  (\ref{bound 2}).
\end{proof}

\section{Event-triggered control}\label{sec:event-triggered control}
In this section,  an event-triggered implementation of the controller from Section \ref{sec:continuous time controller} is developed. That is, the control input is restricted to be piecewise constant with 
\begin{align}
U(t)&=U(t_k) & \text{for all }& t\in[t_k,t_{k+1})
\end{align}
where $t_k$, $k\in\N$, are the update times. 

In the continuous-time implementation of the feedback controller, the control input is updated continuously so that the uncontrolled boundary value is kept  at $v(0,t) = U^*(t) = 0$ for $t\geq \tau^v(0)$. In the absence of any disturbances or prediction errors due to uncertainty, this ensures that the state is kept exactly at the origin. While such idealized assumptions are never exactly  achievable  in practice, it is usually also acceptable if the state remains sufficiently ``close'' to the origin. By the bound in (\ref{bound 2}), one can ensure that $\|w(\cdot,t)\|\leq \tilde{\epsilon}$ for any $\tilde{\epsilon}>0$ and sufficiently large $t$ by keeping $|v(0,t)|\leq \frac{\tilde{\epsilon}}{c_2}$.  Therefore, the idea of the event-triggered control law is to compute the control inputs in the same way as in Theorem \ref{theorem continuous time}, but to only update the control input at times $t_k$ where the prediction  of $|v(0,t_k+\tau^v(0))|$ under the previous control input exceeds some threshold $\varepsilon>0$. 

We propose the following event-triggered feedback controller. At time   $t>0$, given the state $w(\cdot,t)$ and threshold $\varepsilon$,  the control input $U(t)$  is set according to the following algorithm, which is initialized with $\bar{U}=U(0)$ and $t_0=0$. The algorithm also produces the Zeno-free sequence of update times $t_k$, as established in Theorem \ref{theorem event-triggered}.
\begin{algorithm}[H]
\begin{algorithmic}[1]
\STATE Predict $\bar{u}(\cdot,t)$ by solving (\ref{u_t})-(\ref{uBC}) with  state-measurement $w(\cdot,t)$ as initial condition over the domain $\A^u(t)$;
\STATE Predict $v(0,t+\tau^v(0))=\bar{v}^{\circ}(0,t)$ under the current input by solving
\begin{align}
\bar{v}_x^{\circ}(x,t)&= -\frac{f^v((\bar{u}(x,t),\bar{v}^{\circ}(x,t)),x)}{\lambda^v(x)}, \label{vbarcirc_x}\\
\bar{v}^{\circ}(1,t) &= \bar{U}; \label{vbarcirc_BC}
\end{align}
 \IF{$|\bar{v}^{\circ}(0,t)|> \varepsilon$}
	\STATE solve target dynamics (\ref{vbar*_x})-(\ref{vbar*_BC});
	\STATE  set $\bar{U} = \bar{v}^*(1,t)$ and append update times by $t_{k+1}=t$;
	\ENDIF
\STATE Set $U(t) = \bar{U}$.
\end{algorithmic}
\caption{Event-triggered control algorithm }
\label{algorithm}
\end{algorithm}
%
\begin{theorem}\label{theorem event-triggered}
The system consisting of (\ref{u_t})-(\ref{w_0}) in closed-loop with $U(t)$ constructed by the event-triggered control algorithm above satisfies $|v(0,t)|\leq \varepsilon$ for all $t\geq \tau^v(0)$ and $\|w(\cdot,t)\|_{\infty}\leq c_2\,\varepsilon$ for all $t\geq \tau^v(0)+\tau^u(1)$. Moreover, there exists  $\Delta>0$, which depends on $\varepsilon$, $\|w_0\|_{\infty}$ and the model data, such that $t_{k+1}-t_k\geq \Delta$ for all $k\in\N$. 
\end{theorem}
\begin{proof}
The construction is such that if $|\bar{v}(0,t)| > \varepsilon$, then $U$ is updated to a value such that $\bar{v}(0,t)$ becomes zero. The bound on $\|w(\cdot,t)\|_{\infty}$ for $t\geq \tau^v(0)+\tau^u(1)$ follows directly from (\ref{bound 2}). It only remains to be shown that $t_{k+1}-t_k\geq \Delta$. The idea of the proof is to show that for short $\Delta$ the solution to (\ref{vbarcirc_x})-(\ref{vbarcirc_BC}) at $x=0$ is continuous in time.

We first establish an a-priori bound on $\|w\|_{\infty}$. Choosing $\varepsilon< c_1\,\|w_0\|_{\infty}$ and using bound (\ref{bound 1 v}) we have that the closed-loop system satisfies $\|v(0,\cdot)\|_{\infty}\leq c_1\,\|w_0\|_{\infty}$. Thus, via bound (\ref{bound 2}) we obtain that $\|w\|_{\infty}\leq c_1\,c_2\,\|w_0\|_{\infty}$. 

 The control input $\bar{U}$ is first initialized at time $t=0$. Assuming $\bar{U}$ was last updated at time $t_k$,  we show by induction that at least some fixed time $\Delta>0$ elapses before the update condition in line 3 of Algorithm \ref{algorithm} can be triggered. 

We rewrite (\ref{vbarcirc_x})-(\ref{vbarcirc_BC}) as an integral equation with $\bar{u}$ playing the role of a parameter
\begin{equation}
\bar{v}^{\circ}[\bar{u}](x,t) = \bar{U} - \int_1^x \frac{f^v((\bar{u}(\xi,t),\bar{v}^{\circ}(\xi,t)),\xi)}{\lambda^v(\xi)} d\xi. \label{vbar intergral equation}
\end{equation}
Due to the Lipschitz assumptions, the solution to (\ref{vbar intergral equation}) depends continuously on $\bar{u}$, i.e., there exists a constant $c_3$ such that 
\begin{equation}
\|\bar{v}^{\circ}[\bar{u}_1]-\bar{v}^{\circ}[\bar{u}_2] \|_{\infty} \leq c_3\, \|\bar{u}_1 - \bar{u}_2\|_{\infty}  \label{vcirc bound}
\end{equation}

Parameterize the characteristic lines of $u$ passing through a point $(x,t)$ via the solution of the ODE
\begin{align}
\frac{d}{ds}z(x,t;s)&= \lambda^u(z(x,t;s)), & z(x,t;t)&=x.
\end{align}
Then $u$ is continuous in $s$ along the characteristic line $(z(x,t;s),s)$ as long as $z(x,t;s)\in[0,1]$ 
(see \cite[Section A.3]{strecker2017output}). In particular, $u$ satisfies
\begin{equation}
\frac{d}{dt}u(z(x,t_k;t),t) = f^u(w(z(x,t_k;t),t),z(x,t_k;t)).
\end{equation}
 Hence, 
\begin{equation}
|u(z(x,t_k;t),t)-u(x,t_k)| \leq c_4 e^{c_5 (t-t_k)}
\end{equation}
where $c_4$ scales linearly with $\|w\|_{\infty}$. Applying the same approach to the transformed system (\ref{w_bar}) gives
\begin{equation}
|\bar{u}(z(x,t_k;t))-\bar{u}(x,t_k)| \leq c_4 e^{c_5 (t-t_k)}.
\end{equation}

Define $\theta^v(x) = \int_0^x\frac{1}{\lambda^v(\xi)}d\xi$.
Also define $y_1(t-t_k)$ as the solution to the equation $\theta^v(y_1) = t-t_k$, and $y_2(t-t_k)$ as the solution to $\tau^v(y_2)=t-t_k$.
 Then
\begin{align}
 |y_1(\Delta)|&\leq\Delta\,k_{\Lambda}&\text{ and }& &|1-y_2(\Delta)|&\leq\Delta\,k_{\Lambda}. \label{bound y1 y2}
\end{align}

 \begin{figure}[htbp!]\centering
 \includegraphics[width=.75\columnwidth]{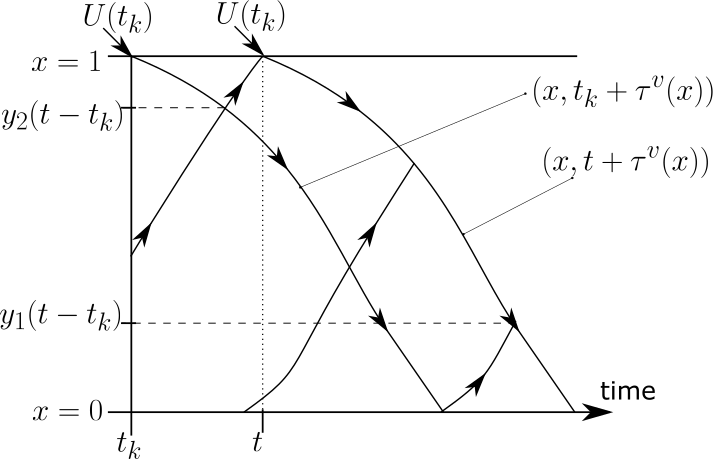}
 \caption{Schematic of the integrations paths and continuity of $u$ along its (upwards) characteristic lines  in the proof of Theorem \ref{theorem event-triggered}.}
 \label{fig:proof event triggered}
 \end{figure}
 
For all $\Delta<\tau^u(1)$, it follows that $z_1(\Delta)<1$ and $z_2(\Delta) > 0$. So for any $x_1\in[z_1(t-t_k),1]$ there exists  $x_2\in[0,z_2(t-t_k)]$ such that $x_1 = z(x_2,t_k;t)$. Using this, the integrals in (\ref{vbar intergral equation}) can be split into different parts  (see also Figure \ref{fig:proof event triggered})  to obtain
\begin{equation}
\begin{aligned}
&|\bar{v}^{\circ}[\bar{u}(\cdot,t)](0,t) - \bar{v}^{\circ}[\bar{u}(\cdot,t_k)](0,t_k) |\\
&= \left|\int_0^{y_2(t-t_k)}  \frac{f^v(\bar{w}(z(x,t_k;t),t),z(x,t_k;t))}{\lambda^v(z(x,t_k;t))}   \right.  \\
&\qquad - \frac{f^v(\bar{w}(x,t_k),x) }{\lambda^v(x)}dx\\
&\quad +  \int_0^{y_1(t-t_k)} \frac{f^v(\bar{w}(x,t),x) }{\lambda^v(x)}dx \\
& \left. \quad -  \int_{y_2(t-t_k)}^1 \frac{f^v(\bar{w}(x,t_k),x) }{\lambda^v(x)}dx  \right|=\left|I_1+I_2+I_3\right|,
 \end{aligned}\label{sum of integrals}
\end{equation}
where $U(t)-U(t_k)=0$ has been used, and $I_1$, $I_2$ and $I_3$ are the first, second and third integrals, respectively. This can be bounded as follows. Due to the Lipschitz condition we have that  
\begin{equation}
\begin{aligned}
|I_1|&\leq l_F k_{\Lambda^{-1}} \left( \|\bar{u}(\cdot,t)-\bar{u}(\cdot,t_k)\|_{\infty}\right. \\
&\hspace{1.8cm}+ \left. \|\bar{v}^{\circ}[\bar{u}(\cdot,t)]-\bar{v}^{\circ}[\bar{u}(\cdot,t_k)]\|_{\infty} \right)\\
&\leq l_F k_{\Lambda^{-1}}(1+c_3\,) \|\bar{u}(\cdot,t)-\bar{u}(\cdot,t_k)\|_{\infty},
\end{aligned}
\end{equation}
where  (\ref{vcirc bound}) is used in the last step. Due to (\ref{bound y1 y2}) and using the Lipschitz conditions and the a-priori bound on $\|w\|$, 
\begin{align}
|I_2|\leq 2\, l_F k_{\Lambda^{-1}} c_1c_2 k_{\Lambda} (t-t_k),\\
|I_3|\leq 2\, l_F k_{\Lambda^{-1}} c_1c_2 k_{\Lambda} (t-t_k).
\end{align}
In summary, the integrals on the right-hand side of (\ref{sum of integrals}) are bounded and
\begin{equation}
|\bar{v}^{\circ}[\bar{u}(\cdot,t)](0,t) - \bar{v}^{\circ}[\bar{u}(\cdot,t_k)](0,t_k) | \leq h(t-t_k)
\end{equation}
with some continuous, strictly increasing function $h$ that satisfies $h(0)=0$ (where $h$ also depends on the a-priori bound on $\|w\|_{\infty}$, and thus, $\|w_0\|_{\infty}$ via the constant $c_4$). Therefore, there exists  $\Delta>0$ such that $h(t-t_k)\leq \varepsilon$ for all $t\leq t_k+\Delta$, which completes the proof.
\end{proof}

\begin{remark}[State-dependent $\varepsilon$]
In Algorithm \ref{algorithm} the trigger-parameter $\varepsilon$ is chosen as a constant.  An alternative implementation would be to choose $\varepsilon$ dependent on the state. The minimum inter-trigger time $\Delta$ decreases with increasing state norm. This could lead to frequent sampling if, e.g., the initial condition is far from the origin. In such situations it would likely be acceptable if the control specifications were relaxed, so that the state is first brought ``closer'' to the origin before the controller is tightened. For instance, if  $\varepsilon=\frac{\gamma}{c_2}\|w(\cdot,t)\|_{\infty}$ with $\gamma<1$, then by (\ref{bound 2}) one has that $\|w(\cdot,t)\|_{\infty}$ reduces by at least $\gamma$ every  $\tau^v(0)+\tau^u(1)$ units of time. That is, one would obtain exponential convergence. Such state dependent switching could also be combined with a minimum fixed $\varepsilon$ so that the system converges exponentially when it is outside a ball of certain size and is then only kept within that ball. The latter version would avoid switches that might be unnecessary if the system is already ``close enough'' to the origin, as implemented in Section \ref{sec:example}.
\end{remark}

\begin{remark}[Periodic execution and robustness]
Instead of continuously measuring the state, performing the predictions in lines 1-2, and evaluating the trigger condition in line 3 of Algorithm \ref{algorithm}, it would be more practicable to perform these steps only periodically at times $\tilde{t}_i = i\,\tilde{\Delta}$ with some sampling period $\tilde{\Delta}>0$. For instance, by choosing $\tilde{\Delta}<\Delta$ sufficiently small   such that one can prove $|\bar{v}(0,\tilde{t}_{i+1}) - \bar{v}(0,\tilde{t}_{i} )|\leq \frac{1}{2}\varepsilon$, and replacing the trigger condition in line 3 by $|\bar{v}^{\circ}(0,\tilde{t}_{i})|> \frac{1}{2}\varepsilon$, it is possible to guarantee that if the update condition is not satisfied at time $\tilde{t}_i$, then the boundary value still satisfies $|\bar{v}(0,t)|\leq  |\bar{v}(0,\tilde{t}_{i})| + \frac{1}{2}\varepsilon\leq \varepsilon$ for all $t\in [\tilde{t}_{i},\tilde{t}_{i+1}]$.

Similarly, one could account for prediction errors due to model uncertainty or disturbances by tightening the trigger condition further, so that the prediction $|\bar{v}^{\circ}(0,t)|$ plus an error term  still remains below $\varepsilon$. The allowable uncertainties would then depend on $\varepsilon$ or, conversely, the achievable $\varepsilon$ would be limited by the given uncertainty. See  the robustness result in \cite{strecker2021quasilinear2x2}.
\end{remark}

\begin{remark}[Reference tracking]
The developments above focus on stabilization of an equilibrium. Alternatively, tracking objectives of the form 
\begin{equation}
v(0,t) = g_{\text{ref}}(t),
\end{equation}
which includes many objectives of the form $\tilde{g}_{\text{ref}}(w(0,t),t)=0$ (substituting $u(0,t)$ by  (\ref{uBC}) and solving for $v(0,t)$), can be solved by replacing the boundary condition of the target system (\ref{vbar*_x})-(\ref{vbar*_BC}) by 
\begin{align}
\bar{v}^*_x(x,t)&= -\frac{f^v((\bar{u}(x,t),\bar{v}^*(x,t)),x)}{\lambda^v(x)}, \\
\bar{v}^*(0,t) &= g_{\text{ref}}(t+\tau^v(0)),
\end{align}
and setting the trigger condition in line 3 of Algoruthm \ref{algorithm} to
\begin{equation}
|\bar{v}^{\circ}(0,t)-g_{\text{ref}}(t+\tau^v(0))|>\varepsilon.
\end{equation}
\end{remark}

\subsection{Closed-form expression for linear systems} \label{sec:linear closed form}
Evaluating the trigger condition requires solving a nonlinear PDE online when predicting $\bar{u}(\cdot,t)$ and then solving an ODE to obtain $\bar{v}^{\circ}(0,t)$. As is the case for the continuous-time state-feedback controller  \cite[Section 3.4]{strecker2017output}, it is possible to express the trigger condition as a simple integral  of kernels (which can be precomputed) weighting the state.

Consider linear systems of the form 
\begin{align}
u_t &= -\epsilon_1(x) u_{x} + c_1(x) v,\label{linear1}\\
v_t & = \epsilon_2(x) v_{x} + c_2(x)u,\\
v(1,t) &= U(t),\\
u(0,t) &= qv(0,t).\label{linear4}
\end{align}
The derivations in \cite[Section 3.4]{strecker2017output} can be modified slightly to show that the trigger condition can be evaluated via
\begin{equation}\begin{aligned}
\bar{v}^{\circ}(0,t) = \bar{U} - \int_0^1 & \Kvu(1,\xi)u(\xi,t) - \Kvv(1,\xi)v(\xi,t) d\xi,\label{ansatz}
\end{aligned}\end{equation}
where the kernels $\Kvu$ and $\Kvv$ are the solution of the PDE system given in  \cite[Equations (24)-(31)]{vazquez2011backstepping}.
At each trigger instance, the control inputs can be updated as
\begin{equation}
\bar{U} = \int_0^1  \Kvu(1,\xi)u(\xi,t_k) + \Kvv(1,\xi)v(\xi,t_k) d\xi.
\end{equation}
In terms of the original backstepping transformation, this amounts to only updating the control input when the boundary value of the target system (see Equations (5)-(8) in \cite{vazquez2011backstepping})  exceeds the threshold $|\beta(1,t)|\leq \varepsilon$.

\section{Numerical examples}\label{sec:example}
\begin{figure}[htbp!]\centering
\includegraphics[width=\columnwidth]{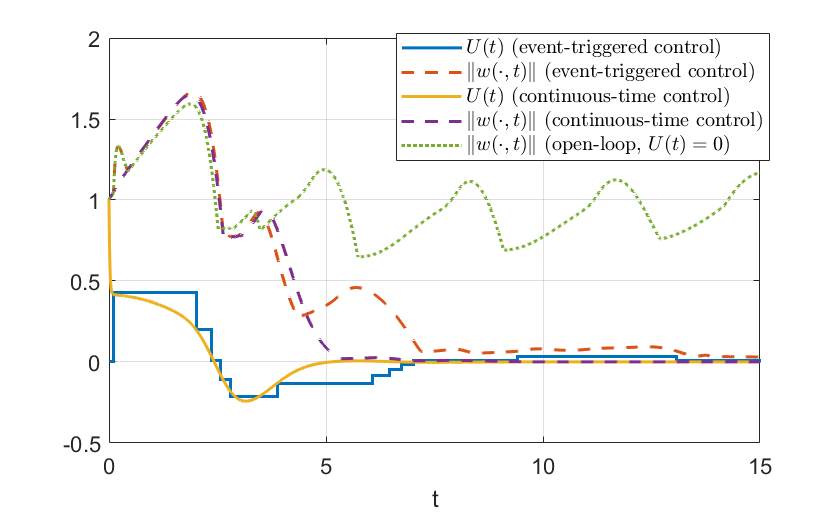}\\
\includegraphics[width=\columnwidth]{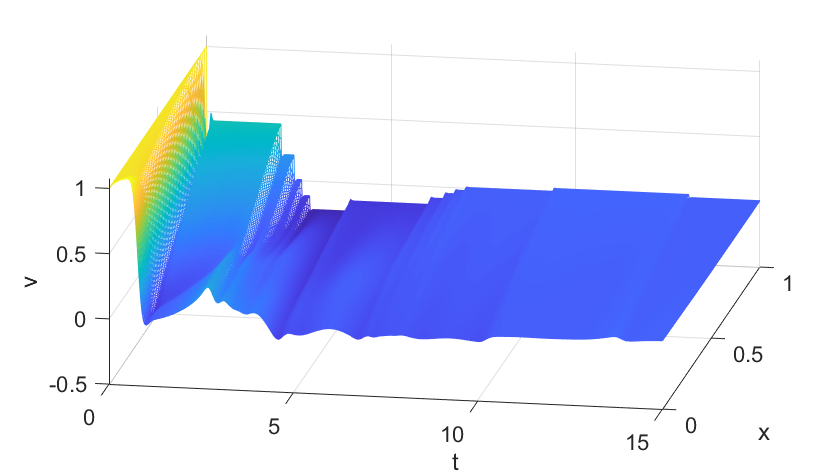}\\
\includegraphics[width=\columnwidth]{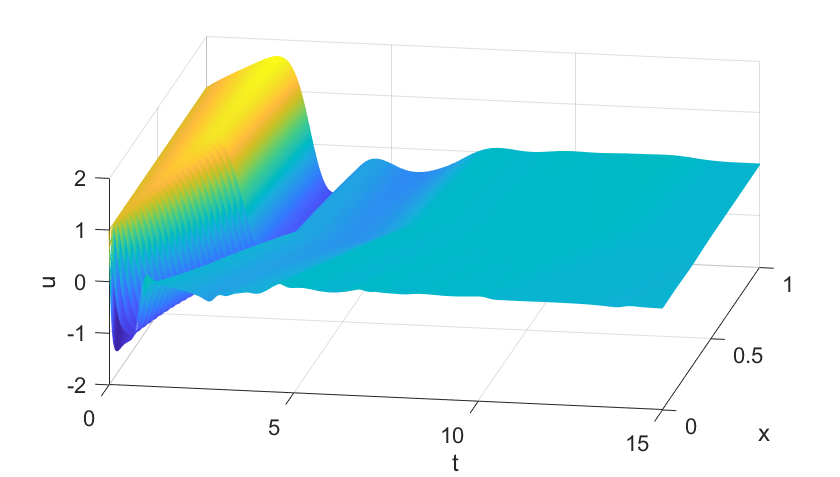}
\caption{Trajectories of the simulation example.}
\label{fig:example}
\end{figure}
The performance of the controller is illustrated below within the context of  numerical simulations of a system with
\begin{align}
\lambda^u(x) &=\begin{cases} 0.2 &\text{if }x<0.5,\\ 2-x &\text{if }x\geq0.5,\end{cases},\\
\lambda^v(x) &=1+0.5\,x,\\
f^u((u,v),x)&=\frac{1}{3-x} \sin(u+v),\\
 f^v((u,v),x) &= \sin(v-u),\\
 g(v,t)&= -v,
\end{align}
and initial condition $u_0(x)=v_0(x)=1$ for all $x\in[0,1]$. The parameters are such that in open-loop with $U(t)=0$ for all $t\geq 0$, the origin is an unstable equilibrium. The simulation and prediction operators are implemented by the method of lines, i.e., semi-discretizing the PDEs in space using first-order finite differences with 50 spatial elements. The resulting high-order ODE is solved in matlab using ode45, where the trigger events are implemented using the odeset('Events',$\ldots$) option. The trigger-parameter is chosen as $\varepsilon = \max\{0.05,0.25\,\|w(\cdot,t)\|_{\infty}\}$. 

The simulated trajectories are shown in Figure \ref{fig:example}. Overall the closed-loop trajectories converge to the origin as expected from the theory.   Convergence is slightly slower when the event-triggered controller is used, as compared to the continuous-time implementation where the control input is updated at every time step of the solver, but performance is still very good.  One can also see the correlation that there are more trigger times  when the continuous control input changes quickly, whereas the control input in the event-triggered version is constant for long periods when the system is close to the origin. In a real-world system this would reduce wear on the actuators significantly.

\section{Conclusion}\label{sec:conclusions}
We proposed an event-triggered implementation of predictive boundary controllers for semilinear hyperbolic systems. A relatively simple static trigger condition is used. The same trigger-mechanism can also be applied to backstepping control of linear systems, where a less computationally expensive implementation is possible. Compared to other approaches to event-triggered control of hyperbolic PDEs in the literature, no Lyapunov functions are required, which avoids some conservativeness. The approach is also generalizable. For instance, by replacing  the state measurement in our event-triggered control scheme   with the state estimate obtained via the finite-time convergent boundary observers from \cite{strecker2017output}, one  solves the event-triggered output feedback control problem. In quasilinear systems, Lipschitz-continuity of the control inputs is necessary for well-posedness of broad solutions. Therefore, when extending the proposed event-triggered control  scheme to quaslinear systems, the transition between different values of $\bar{U}$ needs to be smooth instead of by a jump. Once that transition is completed, the control input can remain constant until the update condition is triggered again.

\bibliographystyle{IEEEtran}
\bibliography{references}      

\end{document}